\documentclass[preliminary,copyright,creativecommons]{eptcs}
\providecommand{\event}{QPL 2011} 
\usepackage{breakurl}             

\usepackage{amssymb,amsthm}

\title{No-go theorems for functorial localic spectra\\ of noncommutative rings}
\author{Benno van den Berg\thanks{Supported by Netherlands Organisation for Scientific Research (NWO).}
\institute{Mathematisch Instituut \\ Utrecht University} \email{B.vandenBerg1@uu.nl}
\and
Chris Heunen\footnotemark[1]
\institute{Department of Computer Science \\ University of Oxford}
\email{heunen@cs.ox.ac.uk}}

\def\titlerunning{No-go theorems for functorial localic spectra of noncommutative rings}
\def\authorrunning{B. van den Berg \& C. Heunen}

\newtheorem{theorem}{Theorem}
\newtheorem{lemma}[theorem]{Lemma}
\newtheorem{proposition}[theorem]{Proposition}
\newtheorem{corollary}[theorem]{Corollary}

\newcommand{\cC}{\ensuremath{\mathcal{C}}}
\newcommand{\Cat}[1]{\ensuremath{\mathbf{#1}}}
\newcommand{\op}{\ensuremath{^{\mathrm{op}}}}
\newcommand{\Mn}[2][n]{\ensuremath{\mathbb{M}_{#1}(#2)}}
\newcommand{\MnC}[1][n]{\Mn[#1]{\mathbb{C}}}
\newcommand{\Cstar}{\Cat{Cstar}}
\newcommand{\Loc}{\Cat{Loc}}
\newcommand{\Top}{\Cat{Top}}

\newcommand{\Idl}{\ensuremath{\mathrm{Idl}}}
\newcommand{\Spec}{\ensuremath{\mathrm{Spec}}}

\begin{document}
\maketitle

\begin{abstract}
  \noindent
  Any functor from the category of C*-algebras to the category of
  locales that assigns to each commutative C*-algebra its Gelfand
  spectrum must be trivial on algebras of $n$-by-$n$ matrices for $n
  \geq 3$. The same obstruction applies to the Zariski, Stone, and
  Pierce spectra. The possibility of spectra in categories other than
  that of locales is briefly discussed.
\end{abstract}

\noindent
A recent article~\cite{reyes:onextensions} by Reyes shows that any
functor $\Cat{CStar}\op \to \Top$ that assigns to each commutative
C*-algebra its Gelfand spectrum must be trivial on the matrix algebras
$\MnC$ for $n \geq 3$. More precisely, it shows that any functorial
extension of the Gelfand spectrum to noncommutative C*-algebras must
yield the empty space on the noncommutative C*-algebras $\MnC$ for $n
\geq 3$. This result shows in a strong way why the traditional notion
of topological space is inadequate to host a good notion of spectrum
for such noncommutative C*-algebras. 

What remains open is whether less orthodox notions of space are also
susceptible to Reyes' theorem. In particular, there are notions of
space, such as that of a locale or a topos, in which the notion of
point plays a subordinate role. Indeed, one of the messages of locale
theory and topos theory is that one can have spaces with a rich
topological structure, but without any points. Initially, this might
seem like an attractive way of circumventing Reyes' result: perhaps
there can be a functor $\Cat{CStar}\op \to \Loc$, extending the Gelfand
spectrum for commutative C*-algebras and assigning nontrivial locales
to noncommutative C*-algebras such as $\MnC$ for $n \geq 3$. 

Unfortunately, the main result of this paper says that this is not the
case: by sharpening Reyes' argument, one can show that any functor
$\Cat{CStar}\op \to \Loc$ which assigns to a commutative C*-algebra its Gelfand
locale must yield the trivial locale (\textit{i.e.}~the locale in which top
and bottom element coincide) on the matrix algebras $\MnC$ for $n \geq
3$. We obtain similar ``no-go theorems'' for (ringed) toposes and even
for quantales (certain structures which have been put forward as a
noncommutative generalization of the notion of locale). Additionally,
we prove similar limitative results for Zariski, Stone and Peirce spectra. 

This does not mean that it is hopeless to look for a good notion of
spectrum of a noncommutative C*-algebra or ring; but what these
result do say is that considerable creativity will be required, in
particular in finding the right generalisation of the notion of
space. We end the paper by mentioning some positive results in
this direction.  

\section{Locale-theoretic preliminaries}

Locales and topological spaces are closely related apart from a few subtle differences. One of
the most important is that in these categories limits are, in general,
computed differently. Initially one might hope that for this reason
Reyes' result does not apply to locales, but it turns out that it does.
The key observation is that, although limits in spaces and 
locales differ in general, they coincide for those spaces (locales)
that arise as spectra. 

\begin{proposition}\label{prop:krlocclosedunderlimits}
  Both compact regular and compact completely regular locales are closed under limits in $\Loc$.
\end{proposition}
\begin{proof}
  The product of compact (completely) regular locales is again (completely) regular and
  compact~\cite[III.1.6, III.1.7, IV.1.5]{johnstone:stonespaces}. The equalizer of $f, g \colon A \to B$ is 
  a closed sublocale of $A$ whenever $B$ is (completely)
  regular~\cite[III.1.3]{johnstone:stonespaces} and a 
  closed sublocale of a compact (completely) regular locale is again
  (completely) regular and
  compact~\cite[III.1.2, IV.1.5]{johnstone:stonespaces}.
\end{proof}

\begin{proposition}
  The limit in $\Loc$ of a diagram of coherent locales and coherent morphisms is 
  coherent. In addition, the mediating morphisms are coherent.
\end{proposition}
\begin{proof}
  A (morphism of) locale(s) is coherent if it lies in the image of the functor
  $\Idl \colon \Cat{DLat} \to \Cat{Frm}$, which is faithful and left adjoint to
  the forgetful functor \cite[II.2.11]{johnstone:stonespaces}.
\end{proof}

\begin{corollary}
  Stone locales are closed under limits in $\Loc$.
\end{corollary}
\begin{proof}
  A locale is Stone when it is both compact regular and coherent. Any continuous morphism
  between Stone spaces is coherent.
\end{proof}

\begin{corollary}
  The functors $\Spec \colon \Cat{cCstar}\op \to \Loc$ and
  $\Idl \colon \Cat{Bool}\op \to \Loc$ which send commutative
  C*-algebras to their Gelfand spectra and  boolean algebras to their
 Stone spectra preserve all limits. 
\end{corollary}
\begin{proof}
  The functor $\mathrm{Spec}$ is part of a duality between commutative
  C*-algebras and compact completely regular locales, and therefore
  certainly preserves all limits as a functor to the category of
  such locales. But as these locales are closed under limits, it also
  preserves all limits when regarded as a functor to the category
  of all locales. The Stone case is completely analogous.
\end{proof}

\section{Main results}

A locale is trivial when it is an initial object in $\Loc$,
\textit{i.e.} when it satisfies $0=1$. 
In categories whose objects contain the matrix rings $\MnC$, let us
call an object $R$ \emph{Kochen--Specker} when there is a   
morphism $\MnC \to R$ for some $n \geq 3$. 
Kochen--Specker objects in the category $\Cat{Ring}$ of rings are
those rings of the form $\Mn{S}$ that allow a ring
homomorphism $\mathbb{C} \to S$~\cite[17.7]{lam:modules}. The class of
Kochen--Specker objects always includes at least $\MnC$ for $n \geq 3$ themselves.
For our purposes it can often be widened; for example, by~\cite{doering:kochenspecker},
Theorem~\ref{thm:vonneumann} below still holds when we include in the class of
Kochen--Specker algebras the von
Neumann algebras without an $\MnC[2]$ factor. 

\begin{lemma}
  If a functor $F \colon \Cat{Ring}\op \to \Loc$ is trivial on $\MnC$
  for all $n \geq 3$, then it is trivial on all Kochen--Specker rings.
\end{lemma}
\begin{proof}
  If $f \colon \MnC \to R$ is a ring morphism, then $Ff \colon FR \to
  F\MnC$ is a locale morphism to the trivial locale, and so $FR$ must be
  trivial. 
\end{proof}

\begin{theorem}\label{thm:cstar}
  Any functor $\Cstar\op \to \Loc$ that assigns to each commutative
  C*-algebra its Gelfand spectrum must be trivial on all
  Kochen--Specker C*-algebras. 
\end{theorem}
\begin{proof}
  For any C*-algebra $A$, let $\cC(A)$ be the diagram of commutative
  C*-subalgebras under inclusion. Define $G(A)$ to be the limit in $\Loc$
  of $\Spec(C)$ with $C \in \cC(X)$. 
  \begin{enumerate}
    \item $G$ is a functor $\Cstar\op \to \Loc$ that assigns to each
      commutative C*-algebra its Gelfand spectrum.
    \item It is the terminal such functor.
    \item Since $\Spec$ preserves limits, $G(A)$ can equally well be computed by first taking the
      colimit of $\cC(A)$ in $\Cat{cCstar}$ and then its Gelfand spectrum.
    \item But for $A = \MnC$ with $n \geq 3$, the colimit of $\cC(A)$ in
      $\Cat{cCstar}$ yields the 0-dimensional C*-algebra; this is the Kochen--Specker
      theorem~\cite{kochenspecker:hiddenvariables}. Hence on these C*-algebras $G$
      yields the trivial locale.
    \item If $F \colon \Cstar\op \to \Loc$ is any other functor that
      assigns to each C*-algebra its Gelfand 
      spectrum, then finality of $G$ guarantees maps $FA \to GA$ for all C*-algebras
      $A$. Hence $FA$ is trivial if $A = \MnC$ for $n \geq 3$.
    \end{enumerate}
    Combining the above observations and the previous lemma yields the
    statement of the theorem. 
\end{proof}

In a similar vein one proves the following
three variations: for Gelfand spectra in the category $\Cat{Neumann}$ of von Neumann
algebras and normal *-homomorphisms; for Stone spectra in the category
$\Cat{PBoolean}$ of partial boolean algebras and partial homomorphisms
(see~\cite{kochenspecker:hiddenvariables,vdbergheunen:colim}); and for
Stone spectra in the category $\Cat{OML}$ of orthomodular lattices
and their homomorphisms. Denote the functor $\Cat{Cstar} \to
\Cat{PBoolean}$ taking projections by $\mathrm{Proj}$.  

\begin{theorem}\label{thm:vonneumann}
  Any functor $\Cat{Neumann}\op \to \Loc$ that assigns to each von Neumann algebra its
  Gelfand spectrum must be trivial on all Kochen--Specker von Neumann algebras.
  \qed
\end{theorem}

\begin{theorem}\label{thm:partialbool}
  Any functor $F \colon \Cat{PBoolean}\op \to \Loc$ that assigns to
  each boolean algebra its Stone spectrum must be trivial on
  $\mathrm{Proj}(\MnC)$ for $n \geq 3$.  
  \qed
\end{theorem}

\begin{theorem}
  Any functor $\Cat{OML}\op \to \Loc$ that assigns to each boolean algebra its Stone
  spectrum must be trivial on $\mathrm{Proj}(\MnC)$ for $n \geq 3$. 
  \qed
\end{theorem}

The Pierce spectrum of a commutative ring, \textit{i.e.} the Stone spectrum of its boolean
algebra of idempotents, requires the following slightly adapted proof.

\begin{theorem}
  Any functor $\Cat{Ring}\op \to \Loc$ that assigns to each commutative ring its Pierce
  spectrum must be trivial on all Kochen--Specker rings.
\end{theorem}
\begin{proof}
  Define the functor $G$ as before: it first considers the diagram $\cC(R)$ of all commutative
  subrings of a ring $R$, and then takes the limit of their Pierce spectra. In addition, define three functors $K, L, M: \Cat{Neumann}\op \to \Loc$: $K$ is simply the composite of $G$ with the inclusion $\Cat{Neumann} \to \Cat{Ring}$, while $L$ on a von Neumann algebra $A$ takes the diagram $\cC(A)$ of commutative von Neumann subalgebras of $A$, and then takes the limit of their Peirce spectra. Finally, $M$ also takes the diagram $\cC(A)$ of commutative von Neumann subalgebras of $A$, but then takes the limit of their Gelfand spectra (which coincides with the Stone spectra on their projections). It is not hard to see that we have natural transformations $K \Rightarrow L$ and $L \Rightarrow M$. But since $M$ results in the trivial locale for $\MnC$ for $n \geq 3$ by Theorem \ref{thm:vonneumann}, the same must be true for $L$, $K$ and $G$. And from here the argument proceeds as before.
\end{proof}

For the Zariski spectrum we argue slightly differently (and
nonconstructively) by reducing the result to the one by Reyes. Let us
emphasize that this proof strategy also applies to the previous
theorems; but whereas they could also be proven constructively, the
Zariski spectrum functor $\Cat{cRing} \to \Loc$ does not preserve limits.

\begin{theorem}
  Any functor $\Cat{Ring}\op \to \Loc$ that assigns to each commutative ring its Zariski
  spectrum must be trivial on all Kochen--Specker rings.
\end{theorem}
\begin{proof}
  Define the functor $G$ as before, and note that we take a limit of a diagram of coherent locales
  and coherent morphisms. As such a limit is coherent and coherent locales are spatial (by the
  prime ideal theorem), its triviality on matrix algebras $\MnC$ for $n \geq 3$ follows from the
  work of Reyes.
\end{proof}

\section{Discussion}

Our main results prove an obstruction to direct functorial
extensions of various spectra, taking values in locales. There is
also no hope for values in categories of which compact
completely regular locales are a subcategory that is closed under limits.
\begin{itemize}
\item The functor $\mathrm{Sh} \colon \Cat{Loc} \to \Cat{Topos}$
  that takes sheaves preserves
  limits~\cite[C.1.4.8]{johnstone:elephant}, so the 
  obstruction for $\Loc$ also holds for $\Cat{Topos}$.
\item The forgetful functor $\Cat{RingedTopos} \to \Cat{Topos}$
  reflects initial objects, so replacing $\Loc$ by the category of
  ringed toposes does not help either. The same holds for ringed
  spaces, either topological or localic.
\item The forgetful functor from the category $\Cat{Scheme}$ of
  schemes to the category $\Top$ of topological spaces reflects initial objects, so
  there is no use in replacing locales by schemes.
\item The category of compact (completely) regular locales is 
  closed under limits in the opposite of the category
  $\Cat{Quantale}$ of unital quantales and their
  homomorphisms~\cite[4.4]{krumletal:max}. Using this adapted version of
  Proposition~\ref{prop:krlocclosedunderlimits}, the proof of
  Theorem~\ref{thm:cstar} also obstructs functors from taking
  values in quantales. Similarly, involutive quantales are out of the question.
\end{itemize}
On the other hand, one could read our main results positively. They
guide the search for a `geometric' spectrum of noncommutative algebras
in two ways. We discuss the Gelfand spectrum here, but the underlying
ways to overcome the obstacle of our main results hold in
general. First, the obstruction can be circumvented by not assigning
the Gelfand spectrum to a commutative C*-algebra \emph{directly}. 
\begin{itemize}
\item Assigning the quantale of closed linear subspaces to a
  C*-algebra encodes the Gelfand spectrum of commutative C*-algebras
  indirectly, and does indeed give a functor~\cite{krumletal:max}. 
\item The Bohrification construction~\cite{heunenlandsmanspitters:bohrification,
  vdbergheunen:colim}, which inspired most of the current work, is not
  a direct extension of the Gelfand spectrum and hence escapes the
  hypothesis of our main theorem. 
\end{itemize}
Secondly, there \emph{is} scope for a functorial spectrum taking values in 
categories with traditional geometric objects but different morphisms.
\begin{itemize}
\item One could consider different morphisms between rings/algebras, and
  hence take a different view of these objects, to obtain a functorial
  spectrum resembling a space (see also the discussion in
  \cite[p15]{reyes:onextensions}). 
\item For example, there \emph{is} an interesting functor $F$ from
  $\Cstar$ to the category of \emph{quantum frames}, that for
  commutative C*-algebras comes down to the Gelfand
  spectrum~\cite{rosicky:quantumframes}. This does not contradict 
  the above results, because there is no forgetful functor from quantum frames
  to either quantales or locales: indeed, $F(\MnC[3])$ consists of
  closed right ideals of $\MnC[3]$, and therefore is not trivial.
\end{itemize}

\bibliographystyle{eptcs}
\bibliography{localicnogo}

\begin{thebibliography}{10}
\providecommand{\bibitemdeclare}[2]{}
\providecommand{\urlprefix}{Available at }
\providecommand{\url}[1]{\texttt{#1}}
\providecommand{\href}[2]{\texttt{#2}}
\providecommand{\urlalt}[2]{\href{#1}{#2}}
\providecommand{\doi}[1]{doi:\urlalt{http://dx.doi.org/#1}{#1}}
\providecommand{\bibinfo}[2]{#2}

\bibitemdeclare{unpublished}{vdbergheunen:colim}
\bibitem{vdbergheunen:colim}
\bibinfo{author}{{B. van den} Berg} \& \bibinfo{author}{C.~Heunen}
  (\bibinfo{year}{2011}): \emph{\bibinfo{title}{Noncommutativity as a
  colimit}}, \doi{10.1007/s10485-011-9246-3}.
\newblock \bibinfo{note}{To appear in \emph{Applied Categorical Structures}.
  Available as arXiv:1003.3618.}

\bibitemdeclare{article}{doering:kochenspecker}
\bibitem{doering:kochenspecker}
\bibinfo{author}{A.~D{\"o}ring} (\bibinfo{year}{2005}):
  \emph{\bibinfo{title}{Kochen-{S}pecker theorem for {V}on {N}eumann
  algebras}}.
\newblock {\sl \bibinfo{journal}{International Journal of Theoretical Physics}}
  \bibinfo{volume}{44}(\bibinfo{number}{2}), pp. \bibinfo{pages}{139--160},
  \doi{10.1007/s10773-005-1490-6}.

\bibitemdeclare{inproceedings}{heunenlandsmanspitters:bohrification}
\bibitem{heunenlandsmanspitters:bohrification}
\bibinfo{author}{C.~Heunen}, \bibinfo{author}{N.P. Landsman} \&
  \bibinfo{author}{B.~Spitters} (\bibinfo{year}{2011}):
  \emph{\bibinfo{title}{Bohrification}}.
\newblock In: {\sl \bibinfo{booktitle}{Deep Beauty---Understanding the Quantum
  World through Mathematical Innovation}}, \bibinfo{publisher}{Cambridge
  University Press}, \doi{10.1017/CBO9780511976971.008}.

\bibitemdeclare{book}{johnstone:stonespaces}
\bibitem{johnstone:stonespaces}
\bibinfo{author}{P.T. Johnstone} (\bibinfo{year}{1982}):
  \emph{\bibinfo{title}{Stone spaces}}.
\newblock {\sl \bibinfo{series}{Cambridge studies in advanced
  mathematics}}~\bibinfo{volume}{3}, \bibinfo{publisher}{Cambridge University
  Press}.

\bibitemdeclare{book}{johnstone:elephant}
\bibitem{johnstone:elephant}
\bibinfo{author}{P.T. Johnstone} (\bibinfo{year}{2002}):
  \emph{\bibinfo{title}{Sketches of an elephant: A topos theory compendium}}.
\newblock \bibinfo{publisher}{Oxford University Press}.

\bibitemdeclare{article}{kochenspecker:hiddenvariables}
\bibitem{kochenspecker:hiddenvariables}
\bibinfo{author}{S.~Kochen} \& \bibinfo{author}{E.~Specker}
  (\bibinfo{year}{1967}): \emph{\bibinfo{title}{The problem of hidden variables
  in quantum mechanics}}.
\newblock {\sl \bibinfo{journal}{Journal of Mathematics and Mechanics}}
  \bibinfo{volume}{17}, pp. \bibinfo{pages}{59--87},
  \doi{10.1512/iumj.1968.17.17004}.

\bibitemdeclare{article}{krumletal:max}
\bibitem{krumletal:max}
\bibinfo{author}{D.~Kruml}, \bibinfo{author}{J.W. Pelletier},
  \bibinfo{author}{P.~Resende} \& \bibinfo{author}{J.~Rosi{c}ky}
  (\bibinfo{year}{2003}): \emph{\bibinfo{title}{On quantales and spectra of
  {C}*-algebras}}.
\newblock {\sl \bibinfo{journal}{Applied Categorical Structures}}
  \bibinfo{volume}{11}(\bibinfo{number}{6}), pp. \bibinfo{pages}{543--560},
  \doi{10.1023/A:1026106305210}.

\bibitemdeclare{book}{lam:modules}
\bibitem{lam:modules}
\bibinfo{author}{T.Y. Lam} (\bibinfo{year}{1999}):
  \emph{\bibinfo{title}{Lectures on Modules and Rings}}.
\newblock {\sl \bibinfo{series}{Graduate texts in mathematics}}
  \bibinfo{volume}{189}, \bibinfo{publisher}{Springer}.

\bibitemdeclare{unpublished}{reyes:onextensions}
\bibitem{reyes:onextensions}
\bibinfo{author}{M.~L. Reyes} (\bibinfo{year}{2011}):
  \emph{\bibinfo{title}{Obstructing extensions of the functor {S}pec to
  noncommutative rings}}.
\newblock \bibinfo{note}{To appear in the \emph{Israel Journal of Mathematics}.
  Available as arXiv:1101.2239.}

\bibitemdeclare{article}{rosicky:quantumframes}
\bibitem{rosicky:quantumframes}
\bibinfo{author}{J.~Rosi{c}ky} (\bibinfo{year}{1989}):
  \emph{\bibinfo{title}{Multiplicative lattices and {C}*-algebras}}.
\newblock {\sl \bibinfo{journal}{Cahiers de topologie et g{\'e}om{\'e}trie
  diff{\'e}rentielle cat{\'e}goriques}}
  \bibinfo{volume}{30}(\bibinfo{number}{2}), pp. \bibinfo{pages}{95--110}.
\newblock \urlprefix\url{http://www.numdam.org/item?id=CTGDC_1989__30_2_95_0}.

\end{thebibliography}

\end{document}